\documentclass[reqno, english]{amsart}
\usepackage{etex}
\usepackage{amsmath,amssymb,amsthm,bbm,mathtools,comment}
\usepackage[shortlabels]{enumitem}
\usepackage[pdftex,colorlinks,backref=page,citecolor=blue]{hyperref}
\hypersetup{pdfpagemode=UseNone,pdfstartview={XYZ null null 1.00}}
\usepackage[mathscr]{euscript}
\usepackage[usenames,dvipsnames]{color}
\usepackage{adjustbox,tikz,calc,graphics,babel,standalone}
\usetikzlibrary{shapes.misc,calc,intersections,patterns,decorations.pathreplacing}
\usetikzlibrary{arrows,shapes,positioning}
\usetikzlibrary{decorations.markings}
\usepackage[final]{microtype}
\usepackage[numbers]{natbib}
\usepackage{cmtiup}
\usepackage{amsfonts}
\usepackage{graphicx}
\usepackage{caption}
\usepackage{subcaption}
\usepackage{verbatim}
\usepackage{array}
\usepackage[frame,cmtip,arrow,matrix,line,graph,curve]{xy}
\usepackage{graphpap, color, pstricks}
\usepackage{pifont}
\usepackage[final]{microtype}
\usepackage{cmtiup}

\allowdisplaybreaks

\theoremstyle{plain}
\newtheorem{theorem}{Theorem}[section]		
\newtheorem{lemma}[theorem]{Lemma}
\newtheorem{claim}[theorem]{Claim}

\theoremstyle{remark}

\let\originalleft\left
\let\originalright\right
\renewcommand{\left}{\mathopen{}\mathclose\bgroup\originalleft}
\renewcommand{\right}{\aftergroup\egroup\originalright}

\makeatletter
\def\imod#1{\allowbreak\mkern10mu({\operator@font mod}\,\,#1)}
\makeatother

\begin{document}
\title{A note on long powers of paths in tournaments}

\author{Ant\'onio Gir\~ao}
\email{tzgirao@gmail.com}

\begin{abstract}
A square of a path on $k$ vertices is a directed path $x_1\ldots x_k$, where $x_i$ is directed to $x_{i+2}$, for every $i\in \{1,\ldots k-1\}$. 
Recently, Yuster showed that any tournament on $n$ vertices contains a square of a path of length at least $n^{0.295}$. In this short note, we improve this bound. More precisely, we show that for every $\varepsilon>0$, there exists $c_{\varepsilon}>0$ such that any tournament on $n$ vertices contains a square of a path on at least $c_{\varepsilon}n^{1-\varepsilon}$ vertices.  
\end{abstract}

\maketitle

\section{Introduction}
Throughout the last decades, the study of sufficient conditions for a graph or an oriented graph to contain certain subgraphs has seen numerous developments, and probably the most important problem concern the existence of a Hamiltonian cycle or a spanning tree (see e.g. \cite{kelly,KomSarSzeme1,JoosKimkuhnosthus,KuhnOsthusSurv}).

In this note, we are interested in finding a long square of path in a tournament.
A $k$th power of a (directed) cycle (or directed path) on $m$ vertices is a sequence of vertices $x_1x_2\ldots x_m$ where the edge $(x_i,x_j)$ is present (and $x_i \rightarrow x_j$, in case of directed path or cycle), for every $i<j \leq i+k$ (we take $i,j$ modulo $m$ in case of a cycle). 
Confirming a longstanding conjecture of Seymour, Koml\'os,  S\'ark\"ozy and, Szemer\'edi \cite{KomSarSzeme}, showed in $1998$, that for every $k$, a sufficiently large graph $G$ on $n$ vertices contains the $k$th-power of a cycle provided $\delta(G)\geq \frac{kn}{k+1}$ and this bound is easily seen to be tight. 

In oriented graphs, such extremal questions are usually harder to answer. Only a decade ago, Keevash, K\"uhn and, Osthus \cite{KeeKuhnOsthus} proved that any sufficiently large oriented graph on $n$ vertices with minimum semidegree greater than $\frac{3n-4}{8}$ is Hamiltonian which is tight. 
We turn now to the problem we address in this note. A tournament is a complete oriented graph. It is easy to show that a tournament on $n$ vertices with minimum semidegree $n/4$ is strongly connected hence it must contain a Hamiltonian cycle. In $1990$, Bollob\'as and H\"aggkvist\cite{bollobashagg} improved this by showing that the same asymptotic bound is enough to guarantee a power of a Hamilton cycle. More precisely, they showed that for $\varepsilon>0$ and $k\in \mathbb{N}$, a tournament with semidegree at least $(1/4+\varepsilon)n$ contains the $k$th power of a Hamilton cycle. 
 Given a tournament $T$, let $pp(T)$ be the longest square of a path in $T$. 
Motivated by the observation that any tournament contains a Hamiltonian path, Yuster \cite{Yuster} studied the following question. What is the smallest value of $pp(T)$ over all tournaments on $n$ vertices?
Very recently, he showed that $pp(T) \geq n^{0.295}$, for every tournament $T$ on $n$ vertices. In here, we substantially improve this bound. 

\begin{theorem}
For every $\varepsilon>0$, there exists $c_{\varepsilon}$ such that the following holds. Let $T$ be a tournament on $n$ vertices then it contains a square of a path of order at least $c_{\varepsilon}n^{1-\varepsilon}$. 
\end{theorem}

\subsection{Notation}
Our notation is standard. Let $T$ be a tournament and $S\subset T$, then we denote $N^{+}(S)=\cap_{x\in T} N^{+}(x)$. Let $(A,B)$ be a balanced bipartite tournament with $|A|=|B|=n$. We denote by $\overset{\rightarrow}{d}[A,B]$ the density of the graph formed by the edges going from $A$ to $B$.

As usual, we say $(A,B)$ form an $(\varepsilon,d)$-regular pair if $\overset{\rightarrow}{d}[A,B]=d$ and for every subset $A'\subset A,B'\subset B$ where $|A'|,|B'|\geq \varepsilon n$, $\overset{\rightarrow}{d}[A',B']=(d \pm \varepsilon)$. 

Let $0<\varepsilon<1$. We say $T$ has a $\varepsilon$-regular partition on $M$ parts if the following hold. 
\begin{enumerate}[i)]
    \item $V(T)=V_0\cup V_1\ldots \cup V_{M}$,
    \item $|V_0|\leq \varepsilon n$, and we shall call it the exceptional set,
    \item $|V_i|=m$, for every $i\in [M]$,
    \item all but at most $\varepsilon M^2$ pairs of parts form $\varepsilon$-regular pairs. 
\end{enumerate}
 
Given an $(\varepsilon,d)$-regular pair of a bipartite tournament $(A,B)$, we say a pair $(x,y)\in A^{(2)}$ is \textit{good} if $$|N^{+}(x)\cap N^{+}(y)|\geq (d^2-10\varepsilon)n.$$

Finally, we write $0<c_1\ll c_2\ll\ldots\ll c_r$, to mean that we can choose the constants $c_1,c_2,\ldots c_k$ from right to left. More precisely, there are non-decreasing functions $f_1, f_2,\ldots, f_{k-1}$ such that, given $c_k$, whenever we choose some $c_i \leq f_i(c_i+1)$, all calculations needed using these constants are valid.

\subsection{Preliminaries}
Here we collect some results we need for our proof. 

\begin{lemma}\label{lem:ord}
Let $G$ be an oriented graph on $n$ vertices without a path of length $k$, then there is an ordering of the vertex set $x_1,\ldots x_n$ such that for every $i\in \{1,\ldots n\}$, $|N^{-}(x_i)\cap \{x_{i+1},\ldots x_n\}|\leq k-1$. 
\end{lemma}
\begin{proof}
This easily proved by induction. Indeed, if all vertices have in-degree at least $k$ we clearly can find a directed path of length $k$, a contradiction. Therefore, there is a vertex $x\in V(G)$ with $d^{-}(x)\leq k-1$. Let $x\coloneqq x_1$ and apply induction to $G-x$.
\end{proof}

\begin{lemma}[Szemer\'edi regularity Lemma \cite{Szmeredireg}]\label{lem:regularity}
For every $\varepsilon>0$ and positive integer $m$, there exists $M=M(\varepsilon,m)$ such that any tournament has a $\varepsilon$-regular partition on $\ell$ parts where $m\leq \ell\leq M$. 
\end{lemma}

\begin{lemma}\label{lem:goodpair}
Let $0<1/n\ll \varepsilon \ll \delta\leq 1/2$. Let $(A,B)$ be a $(\varepsilon,d)$-regular pair where $\delta \leq d$ and $|A|=|B|=n$. Then, for every set $F\subseteq A$ (or $F\subseteq B$), where $|F|\geq \delta^2/4 n$, there is at least a \textit{good} pair $(x,y)$ with $x,y\in F$. 
\end{lemma} 
\begin{proof}
It suffices to who there are at most $10\varepsilon\binom{n}{2}$ pairs of vertices which are not \textit{good} within $A$.
Let $A'\coloneqq\{x\in A \mid d^{+}_{B}(x)\notin \left ](d-2\varepsilon)n,(d+2\varepsilon)n \right [ \}$, by assumption $|A'|\leq 2\varepsilon n$. Let $A^*=A\setminus A'$. Fix a vertex $y\in A^* $, and let $N\coloneqq N^{+}(x)$. By construction, $|N|=(d\pm 2\varepsilon)n\geq \varepsilon n$. Let $ A_1^*\coloneqq \{x\in A^* \mid d^{+}_{N}(x)\notin ](d^2-10\varepsilon)n, (d^2+10\varepsilon)n [\}$. Suppose for contradiction $A_1^{*}\geq 2\varepsilon n$, then we may pass to a subset $A_2^*$ where $|A_2^{*}|\geq \varepsilon n$ with the property that all vertices in $A_2^*$ have at least $(d^2+10\varepsilon)n $ out-neighbours in $N$ (or at most $(d^2-10\varepsilon)n$, the argument is the same in this case).
As every vertex of $A_2^*$ sends at least  $(d^2+10\varepsilon)n$ out-edges to $N$, we have $\overset{\rightarrow}{d}[A_2^*,N]\geq \frac{|A_2^*|(d^2+10\varepsilon)n}{|A_2^*||N|}> (d+\varepsilon)$, which is a contradiction as $|N|,|A_2^*|\geq \varepsilon n$. This implies there are at most $|A^*|2\varepsilon n$ \textit{bad} edges within $A^*$. 
Finally, note that a \textit{bad} pair must either be incident with $A'$ in which case there are at most $|A'|n\leq \varepsilon n^2$ such edges. Or it must lie within $A^*$ for which there are at most $2\varepsilon n^2$ such edges.
\end{proof}

We are now ready to begin the proof of the main theorem. 

\section{Main proof}

\begin{proof}
Let $c_{\varepsilon}\ll 1/m\ll \varepsilon' \ll \delta\ll \varepsilon $. The proof will be by induction on $n$, if $n<c_{\varepsilon}^{-1-\varepsilon}$, there is nothing to show. Let $T$ be a tournament on $n\geq n_0$ vertices.
 First, we shall apply Sz\'em\'er\'edi regularity lemma with parameters $\varepsilon', m$. By Lemma~\ref{lem:regularity}, we can split $T$ into $\ell$ parts, where $m \leq \ell=M(\varepsilon',m)$, forming an $\varepsilon'$-regular partition. Let $V(T)=V_0\cup V_1\ldots \cup V_{\ell}$ be the $\varepsilon$-regular partition. 
  For technical reasons, we add any part $V_i$ (for $i\geq 1$) which is incident with than $2\varepsilon'^{1/2}\ell$ non-regular pairs to $V_0$. Clearly, there are at most $\varepsilon'^{1/2}n$ such parts so we may find a subset of the parts of size at least $\ell'\coloneqq (1-\varepsilon'^{1/2})\ell$ where every part is incident with at most $3\varepsilon^{1/2}\ell'$ non-regular edges. Let $A_0,A_1,\ldots,A_{\ell'}$ be the new partition, where $A_0$ is the new exceptional set. Note that $|A_0|\leq 3\varepsilon'^{1/2}n$.

 \begin{claim}
Suppose there is an $(\varepsilon',d)$-regular pair $(A,B)$, with $\delta \leq d\leq 1-\delta$, then we may find a square of path of length $\frac{\delta^2n}{\ell }\geq c_{\varepsilon}n^{1-\varepsilon}$.
 \end{claim}
 \begin{proof}[Proof of the claim]
Let $(A,B)$ be a $(\varepsilon',d)$-regular pair where $\delta \leq d\leq 1-\delta$. We may iteratively construct a long square of path. The idea is to find a long sequence of edges $e_1,e_2,\ldots e_t$, where $e_{2i}\in A^{(2)}$ and $e_{2i+1}\in B^{(2)}$, for all $i\in [k/2]$ which satisfy the following three properties.
\begin{enumerate}[(i)]
    \item For all $i\in [k]$, $e_{i+1}\subseteq N^{+}(e_{i})$,
    \item For all $i\in [k]$, $e_{i}$ is a \textit{ good} edge,
     \item For all $i\neq j\in [k]$, $e_{i}\cap e_{j}=\varnothing$.
\end{enumerate}
Suppose we have constructed such a sequence $e_1,\ldots e_{2t}$ and we would like to construct $e_{2t+1}$. Note that by assumption $e_{2t}$ is \textit{good} so $|N^{+}(e_{2t})|\geq \delta n/(2\ell)$. We may assume $t\leq \frac{\delta^{2} n}{\ell}$, otherwise we would be done. Let $F\coloneqq N^{+}(e_{2t})\setminus V(\bigcup_{i=1}^{2t}e_i)$, clearly $|F|\geq \delta^2 n/\ell$ and hence by Lemma~\ref{lem:goodpair}, there is a good pair $e_{2t+1}$ within $F$.
It is not hard to see that from a sequence $e_1,\ldots e_t$ as defined above, we may construct a square of a path of length $2t$. Indeed, we may assume $e_i=(x_i,y_i)$, where $x\rightarrow y$. It is not hard to check that $x_1y_1x_2y_2\ldots x_ty_t$ forms a square of path, and we conclude the proof of the claim. 

\end{proof}

 From now on, we may and shall assume that for any $A_i,A_j$ if the pair $(A_i,A_j)$ is $\varepsilon$-regular then either $\overset{\rightarrow}{d}[A_i,A_j]\geq 1-\delta$ or $\overset{\rightarrow}{d}[A_i,A_j]\leq \delta$. We construct an auxiliary oriented graph $D$ where the vertex set is the set $A_1,\ldots A_{\ell'}$ and we add an edge from $A_i$ to $A_j$ if the pair $(A_i,A_j)$ is $(\epsilon,1-\delta)$-regular.
 
 \begin{claim}
Suppose there is a directed path in $D$ of length at least $\delta \ell/2$, then $T$ contains a square of a path of length $c_{\varepsilon}n^{1-\varepsilon}$.
 \end{claim}
 \begin{proof}
 Let $P=A_1\ldots A_k$ be a directed path in $D$, where $k\geq \delta \ell/2$. Now, let $R_{k-1}\subset A_{k-1} $ be the set of vertices which which do not sent at least $(1-\delta-\varepsilon')n/\ell$ out-neighbours to $A_k$. From the fact $(A_{k-1},A_k)$ is an $\varepsilon'$-regular pair, we deduce $|R_{k-1}|\leq \varepsilon' n/\ell$. Remove these vertices from $A_k$, and let $A'_{k-1}\coloneqq A_{k-1}\setminus R_{k-1}$. By assumption, $|A'_{k-1}|\geq (1-\varepsilon')|A_{k-1}|$. 
 Similarly, let $R_{k-1}\subseteq A_{k-2}$ be the set of vertices in $A_{k-2}$ which send less than $(1-\delta-\varepsilon')|A'_{k-1}|$ out-neighbours to $A'_{k-1}$. Again, we know $|R_{k-2}|\leq (1-\varepsilon')|A_{k-2}|$. We may continue in the same fashion all the way down to $A'_1$.
 By induction, we may find a square of a path $H_1\subset A'_1$ of size $c_{\varepsilon}(|A_1|/2)^{1-\varepsilon}$. Let $y,z$ be the last two vertices of $H_1$ and let $B'_2\coloneqq N^{+}(x)\cap N^{+}(z)\cap A'_2$, by assumption $|B'_2|\geq (1-2\delta-3\varepsilon')|A_2|\geq |A_2|/2$. And again by induction, we may find a complete hop $H_2\subset B'_2$ where $|H_2|\geq c_{\varepsilon} (|A_2|/2)^{1-\varepsilon}$. Continuing in a similar way, we may construct a sequence of square of paths $H_1,\ldots H_k$. By construction, they can be put together to form a longer square of a path $H$ of size $\delta \ell /2 \cdot c_{\varepsilon} (n/2\ell)^{1-\varepsilon}\geq c_{\varepsilon}n^{1-\varepsilon} \delta \ell^{\varepsilon}/2^{2-\varepsilon}\geq c_{\varepsilon}n^{1-\varepsilon}$, as we wanted to show. 
 \end{proof}
 
Therefore, we may assume there is no directed path of length $\delta \ell /2$ in $D$. In particular, by Lemma~\ref{lem:ord}, there must be an ordering of the vertices of $D$, $A_1,\ldots A_{\ell'}$ where $|N_{D}^{-}(A_i)\cap \{A_{i+1},\ldots A_{\ell'}\}|\leq \delta \ell' /2$. 
Let $L\coloneqq\bigcup_{i=1}^{\ell'/2}A_i$ and $R\coloneqq\bigcup_{j=\ell'/2}^{\ell'}A_j $.
For technical reasons, we will need to remove few vertices from $L$. Let $A_j\in L$, we say a vertex $x\in A_j$ is \textit{weak} if there is a set $W(x)\subset R$ of size at least $2\varepsilon'^{1/2}\ell'$, for which for all $A_r\in R$, the pair $(A_j,A_{r}$ is $(\varepsilon',d)$-regular for some $d\geq 1-\delta$ but $|N^{+}(x)\cap A_r|\leq (1-2\delta)n/\ell$. By assumption, for a fixed pair $(A_i,A_{j})$ with $A_i\in L$ and $A_j\in R$, forming an $(\varepsilon,d)$-regular pair with $d\geq (1-\delta)$, there are at most $\varepsilon' n/\ell$ vertices which do not send at least $(1-2\delta)n/\ell$ out-neighbours to $A_j$. For a vertex $A_i\in L$, denote by $A_i^{w} $ the subset of $A_i$ consisting of \textit{weak} vertices.

The following holds. 
\begin{align*}
|A_i^{w}| 2\varepsilon'^{1/2}\ell' &\leq \sum_{x \in A_i^{w}} |W(x)|\\
&\leq \sum_{x\in A_i} |W(x)|\leq |R|\varepsilon' n/\ell\leq (n/2\ell'+1)\cdot \varepsilon' n/\ell \implies\\
|A_i^{w}|&\leq \varepsilon'^{1/2}\cdot (n/\ell').
\end{align*}

For every $i\in [\ell'/2]$, add the sets $A_i^{w}$ to the exceptional set $A_0$, and let $L'\coloneqq \bigcup_{i\in [\ell'/2]} A_i\setminus A_i^{w}$. Observe that $|L'|\geq  (1-\delta)n/2$.

\begin{claim}\label{claim:1}
For every $x\in L'$, $|N^{+}_{R}(x)|\geq \left ((1-\delta-\varepsilon'^{1/2})\ell'/2\right ) \cdot (1-2\delta)n/\ell'\geq (1-10\delta)|R|$.
\end{claim}
\begin{proof}
Let $x\in A_i$, for some $i\in [\ell'/2]$. By construction, there are at most $\delta \ell'/2 + \varepsilon^{1/2}\ell'$ parts in $R$ which do not form an $(\varepsilon', d)$-regular pair with $A_i$, for some $d\geq 1-\delta$. 
Moreover, since $x$ is not \textit{weak}, there are at most $2\varepsilon^{1/2}\ell'$ parts in $R$ for which the out-degree of $x$ is smaller than $(1-2\delta)n/\ell'$, hence the claim follows. 
\end{proof}

Now, by induction, we may find a square of a path $H_l\subset V(L')$ of length at least $c_{\varepsilon}|L'|^{1-\varepsilon}$. Let $x,y$ be the last two vertices of $H_l$. By Claim~\ref{claim:1}, $|N^{+}(x)\cap N^{+}(y)\cap R|\geq (1-20\delta)|R|$. Let $N'\coloneqq N^{+}(x)\cap N^{+}(y)\cap R|$. Once again, by induction we may find a square of a path $H_r\subset R$ of size $c_{\varepsilon} {\left(1-20\delta)|R|\right)}^{1-\varepsilon}$. Putting both $H_l$ and $H_r$ together, we obtain a square of a path of size $$2c_{\varepsilon}{\left ((1-30\delta)n/2\right)}^{1-\varepsilon}\geq c_{\varepsilon}n^{1-\varepsilon},$$ 
 the last inequality holds provided $\delta \ll \varepsilon$.

\end{proof}

 \section{Concluding remarks}
We remark that our constant $c_{\varepsilon}$ depends on the application of Szemer\'edi's regularity lemma. It would be nice to obtain a more feasible constant using other methods. We note as well, as pointed out by Yuster, that we still could not rule out the possibility there always exists linear a square of path of length $\Omega(n)$ in any tournament on $n$ vertices.

Lastly, we observe that our arguments can be adapted to prove the existence of long $k$th powers of paths in every tournament. 
\begin{theorem}
For every $0<\varepsilon\leq 1$ and a positive integer $k$, there exists a constant $c_{\varepsilon,k}>0$ such that every tournament on $n$ vertices contains a $k$th power of a path of order at least $c_{\varepsilon,k}n^{1-\varepsilon}$. 
\end{theorem}

\bibliographystyle{amsplain}
\bibliography{squarepaths.bib}
\end{document}